\documentclass{amsart}

\usepackage[headings]{fullpage}

\usepackage{lmodern}
\usepackage[T1]{fontenc}

\usepackage{amsmath}
\usepackage{amsthm}
\usepackage{amssymb}
\usepackage{amsfonts}
\usepackage{latexsym}
\usepackage{mathrsfs}

\usepackage[all]{xy}

\usepackage{paralist}

\usepackage{color}
\usepackage{ifthen}
\usepackage[normalem]{ulem}

\usepackage[colorlinks, pagebackref=true, citecolor=blue, urlcolor=blue,%
linkcolor=blue, pdftex]{hyperref}
\usepackage[alphabetic,abbrev,backrefs,lite]{amsrefs}

\makeatletter

\@addtoreset{equation}{section}
\def\theequation{\thesection.\@arabic \c@equation}
\def\theenumi{\@roman\c@enumi}

\makeatother

\theoremstyle{plain}
\newtheorem{theorem}[equation]{Theorem}
\newtheorem{lemma}[equation]{Lemma}

\newtheorem{proposition}[equation]{Proposition}

\theoremstyle{definition}

\newtheorem{conjecture}[equation]{Conjecture}

\newtheorem{remark}[equation]{Remark}

\newtheorem{example}[equation]{Example}
\newenvironment{examplebox}[1][]{%
    \begin{example}[#1] \pushQED{\qed}}{\popQED \end{example}}

\newtheorem{definition}[equation]{Definition}

\newtheorem{notation}[equation]{Notation}

\newtheorem{discussion}[equation]{Discussion}

\newtheorem{construction}[equation]{Construction}

\newcommand{\define}[1]{\emph{#1}}

\newcommand{\naturals}{\mathbb{N}}
\newcommand{\ints}{\mathbb{Z}}

\DeclareMathOperator{\reg}{reg}

\DeclareMathOperator{\charact}{char}

\DeclareMathOperator{\tor}{Tor}

\DeclareMathOperator{\image}{Im}

\newcommand{\calA}{{\mathcal A}}
\newcommand{\fraka}{{\mathfrak a}}
\newcommand{\bfa}{{\mathbf a}}
\newcommand{\bfb}{{\mathbf b}}

\newcommand{\calB}{{\mathcal B}}
\newcommand{\bfe}{{\mathbf e}}

\renewcommand{\to}{\longrightarrow}

\title{Betti tables of $p$-Borel-fixed ideals}
\author{Giulio Caviglia}
\address{Department of Mathematics, Purdue University, West Lafayette IN
47907, USA}
\email{gcavigli@math.purdue.edu}

\author{Manoj Kummini}
\address{Chennai Mathematical Institute, Siruseri, Tamilnadu 603103, India}
\email{mkummini@cmi.ac.in}
\thanks{The work of the first author was supported by a grant from the
Simons Foundation (209661 to G.~C.). 
The second author was partially supported by a CMI Faculty Development
Grant.
In addition, both the authors thank 
Mathematical Sciences Research Institute, Berkeley CA,
where part of this work was done,
for support and hospitality during Fall 2012.}


\begin{document}

\begin{abstract} 
In this note we provide a counter-example to a conjecture of K.~Pardue
[Thesis, Brandeis University, 1994.], which asserts that if a monomial
ideal is $p$-Borel-fixed, then its $\naturals$-graded Betti table,
after passing to any field does not depend on the field. More precisely, we
show that,
for any monomial ideal $I$ in a polynomial ring $S$ over the ring $\ints$ of
integers and for any prime number $p$, there is a 
$p$-Borel-fixed monomial $S$-ideal $J$ such that a region of 
the multigraded Betti table of $J(S \otimes_\ints \ell)$ is in one-to-one
correspondence with the 
multigraded Betti table of $I(S \otimes_\ints \ell)$ for all fields $\ell$
of arbitrary characteristic.  There is no analogous statement for
Borel-fixed ideals in characteristic zero.  
Additionally, the construction also shows that there are $p$-Borel-fixed
ideals with non-cellular minimal resolutions.
\end{abstract}

\maketitle

\section{Introduction}
\label{sec:intro}

Let $x_1, \ldots, x_n$ be indeterminates over the ring $\ints$ of integers
and $S =  \ints[x_1, \ldots, x_n]$. Let $p$ be zero or a prime number. For
any field $\Bbbk$, the general linear group $\mathrm{GL}_n(\Bbbk)$ acts on
$S \otimes_\ints \Bbbk$.  Say that a monomial $S$-ideal $I$ is
\define{$p$-Borel-fixed}
if $I(S \otimes_\ints \Bbbk)$ is fixed under the action of the Borel
subgroup of $\mathrm{GL}_n(\Bbbk)$ consisting of all the 
upper triangular invertible matrices over $\Bbbk$ for any infinite 
field $\Bbbk$ of
characteristic $p$. (This definition does not depend on the choice of
$\Bbbk$; see~Proposition~\ref{proposition:pBorel}.)

Let $I$ be any monomial $S$-ideal.
In Theorem~\ref{theorem:mainThm} we will show that 
for any  prime number $p$,
there exists a
(monomial) $S$-ideal $J$ that is $p$-Borel-fixed 
and that, for any field $\ell$, there is a region (independent of
$\ell$) in the 
multigraded Betti table of $J(S \otimes_\ints \ell)$ (as a module over 
$S \otimes_\ints \ell$) that is
determined by the multigraded Betti table of $I(S \otimes_\ints \ell)$. 
This shows that, homologically, the class of Borel-fixed ideals in
positive characteristic is as bad as the class of all monomial ideals. 

There is a combinatorial characterization of $p$-Borel-fixed $S$-ideals; 
see Proposition~\ref{proposition:pBorel}. It follows from this 
characterization that if $I$ is $0$-Borel-fixed, then 
$I(S \otimes_\ints \ell)$ if Borel-fixed for all fields $\ell$,
irrespective of $\charact \ell$; the converse is not true.
The Eliahou-Kervaire complex~\cite{ElKeMinimalReslns90}*{Theorem~2.1} gives
$S$-free resolutions of $0$-Borel-fixed ideals in $S$, which specialize to
minimal resolutions over any field field $\ell$.  
In particular, the $\naturals^n$-graded Betti table (and, hence, the 
$\naturals$-graded Betti table)
of a $0$-Borel-fixed $S$-ideal
remains unchanged after passing to any field.
On the other hand, if we only assume that $I$ is $p$-Borel-fixed, with
$p>0$, then little is known about minimal resolutions of
$I(S \otimes_\ints \ell)$ for some field $\ell$, including when 
$\charact \ell = p$.

A systematic study of Borel-fixed ideals in positive characteristic was
begun by K.~Pardue~\cite{PardueThesis94}.
In positive characteristic, Proposition~\ref{proposition:pBorel} was proved
by him. He gave a conjectural formula for the
(Castelnuovo-Mumford) regularity of principal $p$-Borel-fixed ideals.
A.~Aramova and J.~Herzog \cite{ArHePrincipalpBorel97}*{Theorem~3.2} 
showed that the conjectured
formula is a lower bound for regularity; Herzog and D.~Popescu
\cite{HePoRegpBorel01}*{Theorem~2.2} 
finished the proof of the conjecture by showing that it is also an
upper bound. V.~Ene, G.~Pfister and Popescu \cite{EnPfPoBettipStable00}
determined Betti numbers and Koszul homology of a class of 
Borel-fixed ideals in $\Bbbk[x_1, \ldots, x_n]$, 
where $\charact \Bbbk = p > 0$, which they called `$p$-stable'.

Our main result (Theorem~\ref{theorem:mainThm}) arose in the following way.
It is known that the Eliahou-Kervaire resolution is
cellular~\cite{MerminEKCellular10}.  Using algebraic discrete Morse theory,
M.~{J\"ollenbeck} and V.~Welker constructed minimal
cellular free resolutions of principal Borel-fixed ideals in positive
characteristic~\cite{JoWealgDMT09}*{Chapter~6}; see,
also,~\cite{SinefBorel08}.
We were trying to see whether this extends to more general 
$p$-Borel-fixed ideals when we realized the possibility of the existence of
$p$-Borel-fixed ideals whose Betti tables might depend on the
characteristic. As a corollary of our construction and the result of
M.~Velasco~\cite {VelasNonCW08} that there are monomial ideals with a
non-cellular minimal resolution, we conclude that there are $p$-Borel-fixed
ideals that admit a non-cellular minimal resolutions. 

We remarked earlier that the $\naturals$-graded Betti table of a
$0$-Borel-fixed $S$ ideal remains identical over any field.
Pardue~\cite{PardueThesis94}*{Conjecture~V.4, p.~43}
conjectured that this is true also for $p$-Borel-fixed ideals; 
see Conjecture~\ref{conjecture:pardue} for the statement.
(This conjecture also appears in~\cite {PeStEKResln08}*{4.3}.)
There has been some evidence that the conjecture is true. If $J$ is a
$p$-Borel-fixed $S$-ideal, then the projective dimension of 
$J(S \otimes_\ints \ell)$ is determined by the largest $i$ such that $x_i$
divides some minimal monomial generator of $J$. The regularity of
$J(S \otimes_\ints \ell)$ does not depend on $\ell$~\cite
{PardueThesis94}*{Corollary~VI.9}; this is part of the motivation for
Pardue to make this conjecture. Later, Popescu~\cite {PopescuExtremal05}
showed that the extremal Betti numbers of $J(S \otimes_\ints \ell)$ does
not depend on $\ell$.
However, Example~\ref{examplebox:counterExamplePardue} shows that
the conjecture is not true.

We thank Ezra Miller and the anonymous referees for helpful comments.
The computer algebra system
\texttt{Macaulay2}~\cite{M2} 
provided valuable assistance in studying examples.  

\section{Preliminaries}
\label{sec:prelims}

We begin with some preliminaries on estimating the graded Betti numbers of
monomial ideals and on $p$-Borel-fixed ideals. By $\naturals$ we denote the
set of non-negative integers. When we say that $p$ is a prime number, 
we will mean that $p > 0$. By $\bfe_1, \ldots, \bfe_n$, we mean the
standard vectors in $\naturals^n$.

Let $A$ be an $\naturals^d$-graded polynomial ring (for some integer $d \geq
1$) over a field $\Bbbk$, with $A_{\mathbf 0} = \Bbbk$.
Let $M$ be an $\naturals^d$-graded $A$-module. (All the modules that we
deal with in this paper are ideals or quotients of ideals.)
The \define{$\naturals^d$-graded Betti numbers} of $M$ are
$\beta_{i,\bfa}^A(M) :=  \dim_\Bbbk \tor_i^A(M,\Bbbk)_\bfa$.
The \define{$\naturals^d$-graded Betti table} of $M$ is the element 
$(\beta_{i,\bfa}^A(M))_{i, \bfa} \in \ints^{\naturals \times \naturals^d}$.
For $\bfa = (a_1, \ldots, a_d) \in \naturals^d$, we write $|\bfa| = 
a_1+ \cdots+ a_d$. 

\begin{notation}
\label{notation:coeffIdeals}
Let $A$ be a Noetherian ring and $z$ an indeterminate over $A$. Let $B =
A[z]$; it is a graded $A$-algebra with $\deg z = 1$. For a graded $B$-ideal
$I$, define $A$-ideals $I_{\langle i \rangle} = ((I:z^i) \cap A)$, for all
$i \in \naturals$.
\end{notation}

Note that for all $i \in \naturals$,
$I_{\langle i \rangle} \subseteq I_{\langle i+1 \rangle}$. Moreover, since 
$A$ is Noetherian, 
$I_{\langle i \rangle} = I_{\langle i+1 \rangle}$ for all $i \gg 0$. 

\begin{lemma}
\label{lemma:BettiIAndIzI}
Adopt Notation~\ref{notation:coeffIdeals}.
Suppose that $A$ is a $\naturals^d$-graded polynomial ring 
(for some integer $d \geq 1$) over 
a field $\Bbbk$ of arbitrary characteristic, with $A_{\mathbf 0} = \Bbbk$.
Let $I$ be a graded $B$-ideal (in the natural $\naturals^{d+1}$-grading of
$B$). Then for all $\bfa \in \naturals^d$,
\[
\beta_{i,(\bfa,j)}^B(I) = 
\begin{cases}
0, & \text{if}\; j < 0,\\
\beta_{i, \bfa}^A (I_{\langle 0 \rangle}), & 
\text{if}\; j = 0, \;\text{and}\\
\beta_{i-1, \bfa}^A (I_{\langle j \rangle}/I_{\langle {j-1} \rangle}), & 
\text{otherwise}.
\end{cases}
\]
\end{lemma}

\begin{proof}
Fix $\bfa \in \naturals^d$.
Let $M := I_{\langle 0 \rangle}B \oplus \bigoplus_{l\geq1}
(I_{\langle l \rangle}/I_{\langle l-1 \rangle})\otimes_A 
B(-(\mathbf 0, l))$. We need to prove that 
$\beta_{i,(\bfa,j)}^B(I) = \beta_{i,(\bfa,j)}^B(M)$ for all $i,j$. 
Note that $z$ is a non-zero-divisor on $M$.
Moreover, $M/zM \simeq 
I_{\langle 0 \rangle} \otimes_A (B/zB) \oplus \bigoplus_{l\geq1}
(I_{\langle l \rangle}/I_{\langle l-1 \rangle})\otimes_A 
(B/zB)(-(\mathbf 0, l)) \simeq I/zI$. Therefore
there are two exact sequences
\[
\xymatrix@R=1ex{
0 \ar[r]& I(-(\mathbf 0,1)) \ar[r]^-z& I  \ar[r]& I/zI \ar[r]& 0,  \\
0 \ar[r]& M(-(\mathbf 0,1)) \ar[r]^-z& M  \ar[r]& I/zI \ar[r]& 0.}
\]
The maps
$\tor_i^B(I(-(\mathbf 0,1)),\Bbbk) \stackrel{z}{\to} \tor_i^B(I,\Bbbk)$ 
and
$\tor_i^B(M(-(\mathbf 0,1)),\Bbbk) \stackrel{z}{\to} \tor_i^B(M,\Bbbk)$ 
are zero. Therefore, for all $i$ and for all $j > 0$. 
\begin{equation}
\label{equation:bettiIzI}
\beta_{i,(\bfa,j)}^B(I) + \beta_{i-1,(\bfa,j-1)}^B(I) = 
\beta_{i,(\bfa,j)}^B(I/zI) = 
\beta_{i,(\bfa,j)}^B(M) + \beta_{i-1,(\bfa,j-1)}^B(M).
\end{equation}
Note that outside a bounded rectangle inside $\ints^2$, the functions 
$(i,j) \mapsto \beta_{i,(\bfa,j)}^B(I)$ and 
$(i,j) \mapsto \beta_{i,(\bfa,j)}^B(M)$ take the value zero. Therefore it
follows from~\eqref{equation:bettiIzI} that 
$\beta_{i,(\bfa,j)}^B(I) = \beta_{i,(\bfa,j)}^B(M)$ for all $i,j$.
\end{proof}

\begin{definition}
\label{definition:stretch}
Adopt Notation~\ref{notation:coeffIdeals}. 
Let 
$d = (d_0 < d_1 < \cdots)$ be an increasing sequence of natural numbers. 
Define an operation $\Phi_{d}$ on graded $B$-ideals by setting
$\Phi_{d}(I)$ to be the $B$-ideal generated by 
$\oplus_{i\in\naturals} I_{\langle i \rangle} z^{d_i}$.
\end{definition}

\begin{proposition}
\label{proposition:stretchBetti}
Adopt the hypothesis of Lemma~\ref{lemma:BettiIAndIzI}. Then
\[
\beta_{i,(\bfa,j)}(\Phi_{d}(I)) = 
\begin{cases}
\beta_{i, (\bfa, l)}(I), & \text{if}\; j=d_l \\
0, & \text{otherwise}.
\end{cases}
\]
\end{proposition}

\begin{proof}
This follows immediately by noting that, for all $j \in \naturals$,
$(\Phi_{d}(I))_{\langle j \rangle} = 
I_{\langle l \rangle}$ where $l$ is such that $d_{l} \leq j < d_{l+1}$. (If
$d_0 > 0$, then $(\Phi_{d}(I))_{\langle j \rangle} = 0$ for all $0 \leq j <
d_0$.)
\end{proof}

\subsection*{Borel-fixed ideals}
For the duration of this paragraph and
Proposition~\ref{proposition:pBorel}, assume that $p$ is zero or a positive
prime number.
Given two non-negative integers $a$ and $b$, say that 
$a \preccurlyeq_p b$ if $\binom{b}{a} \neq 0 \mod p$.
Then there is the following characterization of 
Borel-fixed ideals; for positive characteristic, it was proved by 
Pardue~\cite {PardueThesis94}*{Proposition II.4}. For details, 
see~\cite{eiscommalg}*{Section~15.9.3}.
\begin{proposition}[\cite{eiscommalg}*{Theorem~15.23}]
\label{proposition:pBorel}
Let $\Bbbk$ be an infinite field of characteristic $p$.
An ideal $I$ of $\Bbbk[x_1,\dots,x_n]$ is Borel fixed if and only
$I$ is a monomial ideal and for all $i<j$ and for all monomial minimal
generators $m$ of $I$, $(x_i/x_j)^sm \in I$ for all $s \preccurlyeq_p t$
where $t$ is the largest integer such that $x_j^t \mid m$.
\end{proposition}

\begin{conjecture}
[\cite{PardueThesis94}*{Conjecture~V.4, p.~43}]
\label {conjecture:pardue}
Let $p$ be a prime number.
Let $I$ be a $p$-Borel-fixed monomial $S$-ideal.
Then the $\naturals$-graded Betti table of $I(S \otimes_\ints \ell)$ is independent of
$\charact \ell$ (equivalently, $\ell$) for all fields $\ell$ (of arbitrary
characteristic).
\end{conjecture}

\section{Construction}
\label{sec:construction}

Recall that $S =  \ints[x_1, \ldots, x_n]$ and that $I$ is 
a monomial $S$-ideal.
Fix a prime number $p$ and let $\Bbbk$ be any field of characteristic $p$.
We now describe an algorithm that constructs an
$S$-ideal $J$ such that $J(S \otimes_\ints \Bbbk)$ is Borel-fixed.

\begin{construction}
\label{construction:pBorelFixed}
Input: A monomial $S$-ideal $I$.  
Set $i=1$ and $J_0 = I$.

\begin{asparaenum}
\item \label{enum:upperBdForReg}
Pick $r_{i}$ an upper bound for 
$\reg_{(S \otimes_\ints \ell)}(J_{i-1}(S \otimes_\ints \ell))$ that is
independent of the field $\ell$.

\item \label{enum:addNewGen}
Pick a positive integer $e_{i}$ such that $p^{e_{i}} > r_{i}$. 
Let $d = (0 < p^{e_{i}} < 2p^{e_{i}} < 3p^{e_i} < \cdots)$. 
Set $J_i = \Phi_d(J_{i-1} + (x_i^{p^{e_i}}))$
with $A = \ints[x_1, \ldots, x_{i}, x_{i+2}, \cdots, x_n]$, $z = x_{i+1}$ 
and $B=S$ (Definition~\ref{definition:stretch}).
Note that we are adding a large power of $x_i$ but
modifying the resulting ideal with respect to $x_{i+1}$.

\item \label{enum:checkAndRepeat} If $i=n-1$ then set $J =J_{i}$ and
exit, else replace $i$ by $i+1$ and go to Step~\eqref{enum:upperBdForReg}.
\end{asparaenum}
Output: A monomial $S$-ideal $J$.
\end{construction}

Before we state our theorem, we need to identify a region of the
$\naturals^n$-graded Betti table of 
$J(S \otimes_\ints \ell)$ that captures the 
$\naturals^n$-graded Betti table of $I(S \otimes_\ints \ell)$. 
Let $\calA = \{\bfa : |\bfa| \le r_1\}$ 
(with $r_1$ as in Step~\eqref{enum:upperBdForReg}) 
and $\calB = \{\bfb : b_j < p^{e_j}-1\}$.
\begin{theorem}
\label{theorem:mainThm}
The ideal $J$ is $p$-Borel-fixed.
Moreover, there is an injective map $\psi : \calA \to \calB$ such that 
for all fields $\ell$ (of arbitrary characteristic), 
for all $1 \leq i \leq n$, and
for all $\bfb \in \calB$,
\[
\beta_{i,\bfb}^{S \otimes_\ints \ell}(J(S \otimes_\ints \ell)) = 
\begin{cases}
\beta_{i,\psi^{-1}(\bfb)}^{S \otimes_\ints \ell}(I(S \otimes_\ints \ell)) 
, & \text{if}\; \bfb \in \image \psi, \\
0, & \text{otherwise}.
\end{cases}
\]
\end{theorem}

Let us make some remarks about the
construction. In Step~\eqref{enum:upperBdForReg}, we may, for example, take
$r_i$ to be the degree of the least common multiple of the minimal monomial
generators of $J_{i-1}$; that this is a bound for regularity (independent
of characteristic) follows from the Taylor resolution. 
There are stronger
bounds, e.g., the largest degree of a minimal generator of the lex-segment
ideal with the same Hilbert function as $J_{i-1}(S \otimes_\ints \ell)$.
Additionally, one may insert a check at Step~\eqref{enum:checkAndRepeat}
whether $J_i(S \otimes_\ints \ints/p)$ is Borel-fixed using
Proposition~\ref{proposition:pBorel}. The algorithm will, then,
terminate before or at the stage $i=m-1$ where $m = \max \{i : x_i
\;\text{divides a minimal monomial generator of}\; I\}$.

The proofs of Theorem~\ref{theorem:mainThm} and
Proposition~\ref{proposition:nonCellularMinRes} hinge on the following
lemma. See~\cite {eiscommalg}*{Section~A3.12} for mapping cones 
and~\cite{MiStCCA05}*{Chapter~4} for cellular resolutions. 
In the proof of the theorem, we first describe the change in the
$\naturals^n$-graded Betti table at Step~\eqref{enum:addNewGen}. 
Readers familiar with multigraded resolutions will be able to see
that the Betti numbers of $J$ in the region $\calB$ should be the Betti
numbers of the ideal obtained from $I$ by replacing $x_i$ with
$x_i^{p^{e_{i-1}}}$ 
and hence contain information of the Betti numbers of $I$.
For the sake of readability,
we will abbreviate, for monomial $S$-ideals $\fraka$,
$\beta_{i,\bfb}^{S \otimes_\ints \ell}(\fraka(S \otimes_\ints \ell))$ 
by $\beta_{i,\bfb}^{\ell}(\fraka)$ and 
$\reg_{(S \otimes_\ints \ell)}(\fraka(S \otimes_\ints \ell))$  by
$\reg_{\ell}(\fraka)$, 
from here till the end of the proof of theorem.

\begin{lemma}
\label{lemma:mappingCone}
Let $1 \leq j \leq n$ and $\ell$ be any field. 
\begin{asparaenum}
\item \label{enum:mappingConeQuotientIsSat}
$(J_{j-1} :_S x_j^{p^{e_j}}) = (J_{j-1} :_S x_j^\infty)$.
\item \label{enum:mappingConeMinimal}
Let $F_\bullet$ and $F'_\bullet$ be minimal $(S \otimes_\ints \ell)$-free resolutions 
of $(S/J_{j-1}) \otimes_\ints \ell$ and 
$(S/(J_{j-1} :_S x_j^{p^{e_j}})) \otimes_\ints \ell$.\\ Write
$M_\bullet$ for the mapping cone of the comparison map 
$F'_\bullet(-x_j^{p^{e_j}}) \to F_\bullet$ that lifts the injective map
$(S/(J_{j-1} :_S x_j^{p^{e_j}}) (-x_j^{p^{e_j}})
\stackrel{x_j^{p^{e_j}}}
\to S/J_{i-1})\otimes_\ints \ell$. 
Then for each $i$, the set of degrees of homogeneous minimal generators
of $F'_{i}(-x_j^{p^{e_j}})$
is disjoint from that of $F_i$.  In particular,
$M_\bullet$ is a minimal $(S \otimes_\ints \ell)$-free resolution of 
$(S/(J_{j-1} + (x_j^{p^{e_j}})))  \otimes_\ints \ell$. 
\end{asparaenum}
\end{lemma}
\begin{proof}

\underline{\eqref{enum:mappingConeQuotientIsSat}}: Follows from the choice
of ${e_j}$.

\underline{\eqref{enum:mappingConeMinimal}}:
The assertion about generating degrees follows from the choice of ${e_j}$.
As a consequence, we see that the map $F'_i(-x_j^{p^{e_j}}) \to F_i$ is
minimal, i.e., if we represent it by a matrix, all the entries are in the
homogeneous maximal ideal. Therefore $M_\bullet$ is minimal, and, hence a
minimal resolution of 
$(S/(J_{j-1} + (x_j^{p^{e_j}})))  \otimes_\ints \ell$. 
%
\end{proof}

 \begin{proof}[Proof of the theorem]
Without loss of generality, we may assume that $\Bbbk$ is infinite.
Let $x_1^{a_1}\cdots x_n^{a_n}$ be a minimal monomial generator of $J$.
For all $1 \le i \leq n-1$, $a_{i+1}$ is a multiple of 
$p^{e_i}$ and $x_i^{p^{e_i}} \in J$. Note that for all integers $l \geq 1$,
if $m \preccurlyeq_p lp^{e_i}$ for some integer $m$, then $m$ is a multiple
of $p^{e_i}$. By Proposition~\ref{proposition:pBorel} $J$ is
$p$-Borel-fixed; note that $e_1<e_2<\cdots$. 
The assertion about the Betti numbers 
$\beta_{i,\bfb}^{\ell} (J)$ follows from the discussion below, repeatedly
applying~\eqref{equation:dichotomy}. 

Fix $1 \leq j \leq n-1$.
If $|\bfb| \ge i+p^{e_j}$ then $|\bfb| > i + 
\reg_{\ell}(J_{j-1})$, so the Betti numbers
$\beta_{i,\bfb}^{\ell} (J_{j-1} + (x_j^{p^{e_j}}))$ 
are determined by the resolution of 
$(S/(J_{j-1} :_S x_j^\infty))(-p^{e_j}\bfe_j)$; hence, in particular, 
for such $\bfb$, if 
$\beta_{i,\bfb}^{\ell} (J_{j-1} + (x_j^{p^{e_j}})) \neq 0$, then 
$b_j \geq i+p^{e_j}$. 
Putting this together, we obtain the following:
\[
\beta_{i,\bfb}^{\ell}
(J_{j-1} + (x_j^{p^{e_j}})) = 
\begin{cases}
\beta_{i,\bfb}^{\ell} (J_{j-1}), & 
    \text{if}\; b_j < i+p^{e_j} \\
\beta_{i-1,\bfb-p^{e_j}\bfe_j}^{\ell} 
(J_{j-1} :_S x_j^\infty), & \text{otherwise.}
\end{cases}
\]
Proposition~\ref{proposition:stretchBetti} implies that for all
$\bfb \in \naturals^n$,
\begin{equation}
\label{equation:indStepBetti}
\beta_{i,\bfb}^{\ell} (J_{j}) = 
\begin{cases}
\beta_{i,\bfb'}^{\ell} (J_{j-1}), & 
    \text{if}\; p^{e_j} \mid b_{j+1} \;\text{and}\; b_j < i+p^{e_j}, \\
\beta_{i-1,\bfb''}^{\ell} 
(J_{j-1} :_S x_j^\infty), & 
    \text{if}\; p^{e_j} \mid b_{j+1} \;\text{and}\; b_j \ge i+p^{e_j}, \\
0, & \text{otherwise},
\end{cases}
\end{equation}
where write $\bfb' = \bfb - (b_{j+1}-\frac{b_{j+1}}{p^{e_j}}) \bfe_{j+1}$ and
$\bfb'' = \bfb' -p^{e_j}\bfe_j$. We can recover the $\naturals^n$-graded
Betti table of $J_{j-1}$ from the $\naturals^n$-graded Betti
table of $J_j$. To make this precise, suppose that 
$\beta_{i,\bfb}^{\ell} (J_{j}) \neq 0$. Then the resulting dichotomous
situation from~\eqref{equation:indStepBetti} has the following
re-interpretation: 
\begin{equation}
\label{equation:dichotomy}
\begin{aligned}
b_j < i+p^{e_j} 
&\quad\text{if and only if}\quad
\beta_{i,\bfb}^{\ell} (J_{j}) = 
\beta_{i,\bfb'}^{\ell} (J_{j-1}),\\
b_j \ge i+p^{e_j} 
&\quad\text{if and only if}\quad
\beta_{i,\bfb}^{\ell} (J_{j}) = 
\beta_{i-1,\bfb'}^{\ell} (J_{j-1} :_S x_j^\infty).
\end{aligned}
\end{equation}
We will not explicitly construct the map $\psi$, but will observe that it
can be done putting together the changes at each stage $j$.
\end{proof}

\begin{proposition}
\label{proposition:nonCellularMinRes}
Let $p$ be any prime number, $\Bbbk$ a field of characteristic $p$ and
$R := S \otimes_Z \Bbbk = \Bbbk[x_1, \ldots, x_n]$.
Let $I$ be any monomial $S$-ideal and $J$ be as in
Construction~\ref{construction:pBorelFixed}. If $IR$ has a non-cellular
minimal $R$-free resolution then so does $JR$.
In particular, there exists a Borel-fixed 
$R$-ideal with a non-cellular minimal resolution.
\end{proposition}

\begin{proof}
The second assertion follows from the first since there are
monomial ideals that have non-cellular minimal
resolutions~\cite{VelasNonCW08}; therefore we prove that if $IR$ is a
non-cellular minimal resolution then so does $JR$.
As  proposition does not involve looking at the behaviour of $I$ and $J$ in
two different characteristics, so, for the duration of this proof, we may
assume that Construction~\ref{construction:pBorelFixed} is done over $R$
instead of $S$. Hereafter, we assume that $I$ and $J$ are $R$-ideals. 

Note that it suffices to show, inductively, that, in
Construction~\ref{construction:pBorelFixed}, if $J_{i-1}$ has a
non-cellular minimal resolution, then so does $J_i$.
It is immediate that $J_i$ has a cellular minimal resolution if and only if
$(J_{i-1} + (x_i^{p^{e_i}}))$ has one; this is
because the same CW-complex supports minimal resolutions of 
$(J_{i-1} + (x_i^{p^{e_i}}))$ and
$J_i := \Phi_d(J_{i-1} + (x_i^{p^{e_i}}))$. Therefore, it suffices to show
that if $J_{i-1}$ has a non-cellular minimal resolution then so does
$(J_{i-1} + (x_i^{p^{e_i}}))$.

This is an immediate consequence of the choice of $e_i$ and of 
Lemma~\ref{lemma:mappingCone}.
Let $F_\bullet$ be a non-cellular minimal resolution of $J_{i-1}$. 
Let $F'_\bullet$ be any minimal resolution
of $S/(J_{j-1} :_S x_j^{p^{e_j}})$. Then the mapping cone $M_\bullet$ is
necessarily non-cellular: for, otherwise, if there is a CW-complex $X$ that
supports $M_\bullet$, then for $\bfb = (p^{e_i}-1, \ldots, p^{e_i}-1)$, 
$X_{\leq \bfb}$ supports $F_\bullet$.
\end{proof}

\begin{examplebox}[Counter-examples to Conjecture~\ref{conjecture:pardue}]
\label{examplebox:counterExamplePardue}
Note that,
since graded Betti numbers are upper-semicontinuous functions of
characteristic, 
for an $S$-ideal $J$, the $\naturals$-graded Betti table of 
$(J(S \otimes_\ints \ell))$ depends on $\charact \ell$ if and only if 
the $\naturals^n$-graded Betti table depends on $\charact \ell$.
Let $I$ be any monomial $S$-ideal such that its $\naturals^n$-graded Betti
table depends on $\charact \ell$. 
Let $p$ be any prime number and $\Bbbk$ any field of characteristic $p$.
Let $J$ be the ideal from Construction~\ref{construction:pBorelFixed}. 
Then $J(S \otimes_\ints \ell)$ is Borel-fixed while its 
$\naturals^n$-graded Betti table depends on $\charact \ell$. 
As a specific example, we consider the minimal
triangulation of the real projective plane~\cite
{BrHe:CM}*{Section~5.3}. We have
\begin{align*}
S & = \ints[x_1, \ldots, x_6] \\
I & = (x_1x_2x_3, x_1x_2x_4, x_1x_3x_5, x_2x_4x_5, x_3x_4x_5, x_2x_3x_6,
x_1x_4x_6, x_3x_4x_6, x_1x_5x_6, x_2x_5x_6).
\intertext{With $p=2$,
$e_1 = 3$,
$e_2 = 5$,
$e_3 = 7$,
$e_4 = 9$, and
$e_5 = 11$, we obtain}
J & = 
(x_1^8, x_2^{32}, x_1x_2^8x_3^{32}, x_3^{128}, x_1x_2^8x_4^{128},
x_4^{512}, x_1x_3^{32}x_5^{512}, x_2^8x_4^{128}x_5^{512},
x_3^{32}x_4^{128}x_5^{512},  \\
& \qquad 
x_5^{2048}, x_2^8x_3^{32}x_6^{2048}, x_1x_4^{128}x_6^{2048},
x_3^{32}x_4^{128}x_6^{2048}, x_1x_5^{512}x_6^{2048},
x_2^8x_5^{512}x_6^{2048}).
\end{align*}
Then the Betti numbers 
$\beta_{2,2729}^{S \otimes_\ints \ell}(J(S \otimes_\ints \ell))$ and
$\beta_{3,2729}^{S \otimes_\ints \ell}(J(S \otimes_\ints \ell))$ 
(which correspond to 
$\beta_{2,6}^{S \otimes_\ints \ell}(I(S \otimes_\ints \ell))$ and
$\beta_{3,6}^{S \otimes_\ints \ell}(I(S \otimes_\ints \ell))$,
respectively) are nonzero precisely when $\charact \ell = 2$; 
otherwise they are zero.
\end{examplebox}

After this paper was posted on the {\tt arXiv}, Matteo Varbaro asked us
whether there are $p$-Borel-fixed ideals minimally generated in a single
degree that exhibit different Betti tables in different characteristics.
There are: for instance, if we take $J_1$ to be the sub-ideal of the ideal
$J$ of the above example generated by the monomials of degree $2725$ in
$J$, i.e., $J_1 = J \cap (x_1, \ldots, x_6)^{2725}$. Being the intersection
of two $p$-Borel-fixed ideals, $J_1$ is $p$-Borel-fixed. Moreover, for all
$i$, for all $j > 2725$  and for all fields $\ell$,
$\beta_{i,i+j}^{S \otimes_\ints \ell}(J(S \otimes_\ints \ell)) = 
\beta_{i,i+j}^{S \otimes_\ints \ell}(J_1(S \otimes_\ints \ell))$.

\def\cfudot#1{\ifmmode\setbox7\hbox{$\accent"5E#1$}\else
  \setbox7\hbox{\accent"5E#1}\penalty 10000\relax\fi\raise 1\ht7
  \hbox{\raise.1ex\hbox to 1\wd7{\hss.\hss}}\penalty 10000 \hskip-1\wd7\penalty
  10000\box7}

\end{document}